\documentclass[a4paper,12pt]{amsart}

\usepackage{epsfig}
\usepackage{amsthm,amsfonts}
\usepackage{amssymb,graphicx,color}
\usepackage[all]{xy}
\usepackage{verbatim}
\usepackage{hyperref}
\usepackage{enumitem}

\newtheorem{theorem}{Theorem}[section]
\newtheorem*{theorem*}{Theorem}

\newtheorem{proposition}[theorem]{Proposition}
\newtheorem{definition}[theorem]{Definition}

\newtheorem{remark}[theorem]{Remark}

\newcommand{\R}{\mathbb{R}}

\newcommand{\C}{\mathbb{C}}

\newcommand{\N}{\mathbb{N}}

\newcommand{\B}{\mathbb{B}}

\newcommand{\x}{\textbf {x}}

\newcommand{\T}{\widetilde}

\newcommand{\Sing}[1]{{\rm Sing}(#1)}

\begin{document}

\title[Multiplicity of analytic hypersurface singularities under ...]
{Multiplicity of analytic hypersurface singularities under bi-Lipschitz homeomorphisms}

\author{A. Fernandes}
\author{J. Edson Sampaio}
\address{A. ~Fernandes and J. Edson ~Sampaio - 
Departamento de Matem\'atica, Universidade Federal do Cear\'a Av.
Humberto Monte, s/n Campus do Pici - Bloco 914, 60455-760
Fortaleza-CE, Brazil} 

 \email{alexandre.fernandes@ufc.br}
 \email{edsonsampaio@mat.ufc.br}

\thanks{The first named author were partially supported by CNPq-Brazil grant 302764/2014-7}

\begin{abstract}
We give partial answers to a metric version of Zariski's multiplicity conjecture. In particular, we prove the multiplicity of complex analytic surface (not necessarily isolated) singularities in $\C^3$ is a bi-Lipschitz invariant.
\end{abstract}

\maketitle
\setcounter{section}{-1}
\section{Introduction}

Let $f\colon (\C^n,0)\to (\C,0)$ be the germ of a reduced holomorphic function at origin. Let $(V(f),0)$ be the germ of the zero set of $f$ at origin. We recall \emph{the multiplicity of} $V(f)$ at origin, denoted by $m(V(f),0)$, is defined as following: we write
$$f=f_m+f_{m+1}+\cdots+f_k+\cdots$$ where each $f_k$ is a homogeneous polynomial of degree $k$ and $f_m\neq 0$. Then, $m(V(f),0):= m.$

In 1971 (see \cite{Zariski:1971}), O. Zariski proposed the following problems:

\begin{enumerate}[leftmargin=0pt]
\item[]{\bf Question A} If there exists a homeomorphism $\varphi\colon (\C^n,V(f),0)\to (\C^n,V(g),0)$, then $m(V(f),0)=m(V(g),0)$?
\item[]{\bf Question B} If there is a homeomorphism $\varphi\colon (\C^n,V(f),0)\to (\C^n,V(g),0)$, then are the projectivised tangent cones of $V(f)$ and $V(g)$, at the origin, homeomorphic?
\end{enumerate}

In 2005 (see \cite{Bobadilla:2005}), J. Bobadilla gave a negative answer to Question B. In fact, Bobadilla showed the existence of a family of function-germs $f_t\colon (\C^5,0)\to (\C,0)$ with isolated singularities such that: for $t\neq 0$, there exists a homeomorphism $\varphi_t\colon (\C^5,V(f_0),0)\to (\C^5,V(f_t),0)$ but the respective projectivised tangent cones are not homeomorphic. Concerning Question A, although many authors have presented several partial results, this question remains open. In order to know more about the Zariski's multiplicity question see, for example, \cite{Ephraim:1976,Eyral:2007,Gau-Lipman:1983,Greuel:1986,OShea:1987} and \cite{Trotman:1977}.

In this paper, we deal with metric aspects related to the above Zariski's questions. More precisely, we consider the question below:

\begin{enumerate}[leftmargin=0pt]
\item[]{\bf Question  \~A1} If there is a bi-Lipschitz homeomorphism $\varphi\colon(\C^n,V(f),0)\to (\C^n,V(g),0)$, then $m(V(f),0)=m(V(g),0)$?
\end{enumerate}

This metric question was approached by some authors even  in a more general setting, however, as far as we know, it remains still open. For instance, G. Comte, in the paper  \cite{Comte:1998}, proved that the multiplicity of complex analytic germs (not necessarily codimension 1 sets) is invariant under bi-Lipschitz homeomorphism with Lipschitz constant close enough to 1. Notice that the bi-Lipschitz homeomorphisms considered by G. Comte have some restrictions. Another result we would like to distinguish here is the bi-Lipschitz invariance of  the multiplicity of normal complex algebraic (not necessarily codimension 1) surface singularities proved by W. Neumann and A. Pichon in the recent preprint \cite{Neumann-Pichon:2014}. See \cite{Chirka:1989} for a definition of multiplicity for higher codimension analytic germs in $\C^n$.

The aim of the present paper is to  give some partial positive answers for Question \~A1.

In Section \ref{section:preliminaries}, we recall the notion of tangent cone, we list some properties of it and we prove that Lelong numbers are bi-Lipschitz invariant. In Section \ref{section:mainresults}, we prove the main results of the paper. Our first result, namely Theorem \ref{main-theorem}, depends on the next variation of Question \~A1.

\begin{enumerate}[leftmargin=0pt]
\item[]{\bf Question  \~A2} Let $f,g\colon \C^n\to \C$ be irreducible homogeneous polynomials. If there is a bi-Lipschitz homeomorphism $\varphi\colon(\C^n,V(f),0)\to (\C^n,V(g),0)$, then $m(V(f),0)=m(V(g),0)$?
\end{enumerate}

Theorem \ref{main-theorem} says: Question \~A1 has positive answer if, and only if, Question \~A2 has positive answer. Finally, in Subsection \ref{subsection:applications}, we do two applications of Theorem \ref{main-theorem}. The first one, Theorem \ref{application1}, shows that Question \~A1 has positive answer for hypersurface singularities whose tangent cones have all irreducible components with isolated singularity at $0$. The last result, Theorem \ref{application2}, proves the bi-Lipschitz invariance of the multiplicity of complex analytic surface (not necessarily isolated) singularities in $\C^3$.

\bigskip

\noindent{\bf Acknowledgements}. We would like to thank L. Birbrair, V. Grandjean and L\^e D.T. for their interest on this work.

\section{Tangent Cones and Lelong numbers}\label{section:preliminaries}
\subsection{Tangent cones}
In this subsection, we set the exact notion of tangent cone that we will use along the paper and we list some of its properties.

\begin{definition}
Let $A\subset \R^{\ell}$ be a subanalytic set such that $x_0\in \overline{A}$.
We say that $v\in \R^{\ell}$ is a tangent vector of $A$ at $x_0\in\R^{\ell}$ if there are a sequence of points $\{x_i\}\subset A\setminus \{x_0\}$ tending to $x_0$ and sequence of positive real numbers $\{t_i\}$ such that 
$$\lim\limits_{i\to \infty} \frac{1}{t_i}(x_i-x_0)= v.$$
Let $C(A,x_0)$ denote the set of all tangent vectors of $A$ at $x_0\in \R^{\ell}$. We call $C(A,x_0)$ the {\bf tangent cone} of $A$ at $x_0$. In case that $x_0=0$, we denote $C(A,0)$ by $C(A)$.
\end{definition}

\begin{remark} {\rm It follows from Curve Selection Lemma for subanalytic sets that, if $A\subset \R^{\ell}$ is a subanalytic set and $x_0\in \overline{A}$ then the following holds true }
$$C(A,x_0)=\{v;\, \exists\, \alpha\colon[0,\varepsilon )\to \R^{\ell}\,\, \mbox{s.t.}\,\, \alpha(0)=x_0,\, \alpha((0,\varepsilon ))\subset A\,\, \mbox{and}\,\, \alpha(t)-x_0=tv+o(t)\}.$$
\end{remark}

\begin{remark}
{\rm If $A\subset \C^n$ is a complex analytic set such that $0\in A$ then $C(A)$ is the zero set of a set of complex homogeneous polynomials (see \cite{Whitney:1972}, Theorem 4D). In particular, $C(A)$ is the union of complex lines passing through the origin $0\in\C^n$.}
\end{remark}

Another way to present the tangent cone of a subset $X\subset\R^{\ell}$ at the origin $0\in\R^{\ell}$ is via the spherical blow-up of $\R^{\ell}$ at the point $0$. Let us consider the {\bf spherical blowing-up} (at origin) of $\R^{\ell}$ 
$$
\begin{array}{ccl}
\rho\colon\mathbb{S}^{\ell-1}\times [0,+\infty )&\longrightarrow & \R^{\ell}\\
(x,r)&\longmapsto &rx
\end{array}
$$

Note that $\rho\colon\mathbb{S}^{\ell-1}\times (0,+\infty )\to \R^{\ell}\setminus \{0\}$ is a homeomorphism with inverse mapping $\rho^{-1}\colon\R^{\ell}\setminus \{0\}\to \mathbb{S}^{\ell-1}\times (0,+\infty )$ given by $\rho^{-1}(x)=(\frac{x}{\|x\|},\|x\|)$. The {\bf strict transform} of the subset $X$ under the spherical blowing-up $\rho$ is $X':=\overline{\rho^{-1}(X\setminus \{0\})}$. The subset $X'\cap (\mathbb{S}^{\ell-1}\times \{0\})$ is called the {\bf boundary} of $X'$ and it is denoted by $\partial X'$. 

\begin{remark}
{\rm If $X\subset \R^{\ell}$ is a subanalytic set, then $\partial X'=\mathbb{S}_0X\times \{0\}$, where $\mathbb{S}_0X=C(X)\cap \mathbb{S}^{\ell-1}$.}
\end{remark}

We finish this subsection reminding the  invariance of the tangent cone under bi-Lipschitz homeomorphisms obtained in the paper \cite{Sampaio:2014}. This result is somehow a positive answer for a metric version of the Zariski's Question B.

\begin{theorem}[Sampaio \cite{Sampaio:2014}, Theorem 3.2]\label{inv_cones}
Let $X,Y\subset\C^n$ be two germs of analytic subsets. If $\varphi\colon(\C^n,X,0)\to (\C^n,Y,0)$ is a bi-Lipschitz homeomorphism, then there is a bi-Lipschitz homeomorphism $d\varphi\colon(\C^n,C(X),0)\to (\C^n,C(Y),0)$.
\end{theorem}

\subsection{Lelong numbers}\label{section:lelong}

Let $X\subset \C^n$ be a complex analytic set such that $0\in X$. Let $X_1,...,X_r$ be the irreducible components of $C(X)$.
Fix $j\in \{1,...,r\}$. For a generic point $x\in (X_j\cap \mathbb{S}^{2n-1})\times \{0\}$, the number of connected components of $\rho^{-1}(X\setminus\{0\})\cap U$ is constant, where $U$ is a sufficiently small open subset of $\C^n\times \R$ with $x\in U$, and we denote this number by $\kappa_X(X_j)$ (see definition of the $n_j$'s in \cite{Kurdyka:1989}, p. 762). Since, in the case of complex analytic germs, the notions of multiplicity and density coincide (see \cite{Draper:1969}), the numbers $\kappa_X(X_j)$ are the same Lelong numbers $n_j=n(X_j)$ defined by Kurdyka and Raby in \cite{Kurdyka:1989} and 
\begin{equation}\label{kurdyka-raby}
m(X)=\sum\limits_{i=0}^r \kappa_X(X_i)m(X_i).
\end{equation}
As another reference to $\kappa_X(X_j)$ numbers, see \cite{Birbrair:2015}, p. 7.

The following result shows the bi-Lipschitz invariance of the Lelong numbers.

\begin{proposition}\label{multiplicities}
Let $X,Y\subset\C^n$ be germs of complex analytic subsets at $0\in\C^n$, with pure dimension $p=\dim X=\dim Y$, and let $X_1,\dots,X_r$ and $Y_1,\dots,Y_s$ be the irreducible components of the tangent cones $C(X)$ and $C(Y)$ respectively. If there exists a bi-Lipschitz homeomorphism $\varphi\colon(\C^n,X,0)\to (\C^n,Y,0)$, then $r=s$ and, up to a re-ordering of indices,  $\kappa_X(X_j)=\kappa_Y(Y_j)$, $\forall \ j$.
\end{proposition}

Before starting the proof of the proposition, we do a slight digression to remind the notion of inner distance on a connected Euclidean subset.

Let $Z\subset\R^{\ell}$ be a path connected subset. Given two points $q,\tilde{q}\in Z$, we define the \emph{inner distance} in $Z$ between $q$ and $\tilde{q}$ by the number $d_Z(q,\tilde{q})$ below:
$$d_Z(q,\tilde{q}):=\inf\{ \mbox{length}(\gamma) \ | \ \gamma \ \mbox{is an arc on} \ Z \ \mbox{connecting} \ q \ \mbox{to} \ \tilde{q}\}.$$

\begin{proof}[Proof of Proposition \ref{multiplicities}.] Let $S=\{t_k\}_{k\in\N}$ be a sequence of positive real numbers such that 
$$t_k\to 0 \quad\mbox{and} \quad \frac{\varphi(t_kv)}{t_k}\to d\varphi(v)$$ 
where $d\varphi$ is a tangent map of $\varphi$ like in Theorem \ref{inv_cones} (for more details, see \cite{Sampaio:2014}, Theorem 3.2). Since, $d\varphi$ is a bi-Lipschitz homeomorphism, we get $r=s$ and there is a permutation $P\colon\{1,\dots,r\}\to \{1,\dots,s\}$ such that $d\varphi (X_i)=Y_{P(i)}$ $\forall \ i.$ This is why we can suppose $d\varphi(X_i)=Y_i$ $\forall \ i$ up to a re-ordering of indices. Let $$SX=\{(\x,t)\in \mathbb{S}^{2n-1}\times S;\,t\x \in X\}.$$ 
Thus, $\rho^{-1}\circ\varphi\circ\rho \colon SX\rightarrow Y'$ is an injective and continuous map that extends continuously to a map $\varphi'\colon \overline{SX}\to Y'.$
 
For each generic point $x\in \mathbb{S}_0 X_j\times\{0\}$, we know $\kappa_X(X_j)$ is the number of connected components of the set $\rho^{-1}(X\setminus \{0\})\cap B_{\delta }(x)$, for  $\delta>0$ small enough. Then, $\kappa_X(X_j)$ can be seen as the number of connected components of the set $(SX\cap \mathbb{S}^{2n-1}\times \{t_k\})\cap B_{\delta }(x)$, for $k$ large enough.

Let $\pi\colon\C^n\to \C^p$ be a linear projection such that 
$$\pi^{-1}(0)\cap(C(X)\cup C(Y))=\{0\}.$$ 
Let us denote the ramification locus of 
$$\pi_{| X}\colon X\to \C^p \quad \mbox{and} \quad \pi_{| C(X)}\colon C(X)\to \C^p$$ 
by $\sigma(X)$ and $\sigma(C(X))$ respectively.

Given a generic point $v'\in \C^p\setminus (\sigma(X)\cup \sigma(C(X)))$ (generic here means that $v'$ defines a direction not tangent to $\sigma(X)\cup \sigma(C(X))$), let $\eta,\varepsilon >0$ be sufficiently small such that 
$$C_{\eta,\varepsilon }(v')=\{w\in \C^p|\, \exists t>0; \|tv'-w\|<\eta t\}\cap \B_{\varepsilon }(0)\subset \C^p\setminus \sigma(X)\cup \sigma(C(X)).$$

The number of connected components of $\pi^{-1}(C_{\eta,\varepsilon }(v'))\cap X$ is exactly $m(X)$, since $C_{\eta,\varepsilon }(v')$ is simply connected and $\pi\colon  X\setminus \pi^{-1}(\sigma(X))\to \C^p\setminus\sigma(X)$ is a covering map. Then, we get that $\pi|_V\colon V\to C_{\eta,\varepsilon }(v')$ is bi-Lipschitz for each connected component $V$ of  $\pi^{-1}(C_{\eta,\varepsilon }(v'))\cap X$. Therefore, for each $j=1,\dots,r$, there are different connected components $V_{j1},\dots,V_{j\kappa_X(X_j)}$ of $\pi^{-1}(C_{\eta,\varepsilon }(v'))\cap X$ such that $C(\overline{V_{ji}})\subset X_j$, $i=1,...,\kappa_X(X_j)$.

Let us suppose that there is $j\in\{1,\dots,r\}$ such that $\kappa_X(X_j)>\kappa_Y(Y_j)$, it means that, if we consider a generic point $x=(v,0)\in\partial X'\cap X_j\times\{0\}$, there are at least two different connected components $V_{ji}$ and $V_{jl}$ of $\pi^{-1}(C_{\eta,\varepsilon }(\pi(v)))\cap X$ and sequences $\{(x_k,t_k)\}_{k\in \N}\subset \rho^{-1}(V_{ji})\cap SX$ and $\{(y_k,t_k)\}_{k\in \N}\subset \rho^{-1}(V_{jl})\cap SX$ such that $\lim (x_k,t_k)=\lim (y_k,t_k)=x$ and $\varphi'(x_k,t_k),\varphi'(y_k,t_k)\in \rho^{-1}(\T V_{jm})$, where $\T V_{jm}$ is a connected component of  $\pi^{-1}(C_{\eta,\varepsilon }(\pi(d\varphi(v))))\cap Y$.

Since $\varphi(t_kx_k),\varphi(t_ky_k)\in\T V_{jm}$ $\forall$ $k\in \N$ and $V=\T V_{jm}$ is bi-Lipschitz homeomorphic to $C_{\eta,\varepsilon }(\pi(d\varphi(v)))$, we have 
$$\|\varphi(t_kx_k)-\varphi(t_ky_k)\|=o(t_k)$$ 
and
$$d_Y(\varphi(t_kx_k),\varphi(t_ky_k))\leq d_V(\varphi(t_kx_k),\varphi(t_ky_k))=o(t_k).$$
Now, since  $X$ is bi-Lipschitz homeomorphic to $Y$, we have $d_X(t_kx_k,t_ky_k)\leq o(t_k)$.
On the other hand, since $t_kx_k$ and $t_ky_k$ lie in different connected components of $\pi^{-1}(C_{\eta,\varepsilon }(\pi(v)))\cap X$, there exists a constant $C>0$ such that $d_X(t_kx_k,t_ky_k)\geq Ct_k$, which is a contradiction.

We have proved that $\kappa_X(X_j)\leq \kappa_Y(Y_j)$, $j=1,...,r$. By similar arguments, using that $\varphi^{-1}$ is a bi-Lipschitz map, we also can prove $\kappa_Y(Y_j)\leq \kappa_X(X_j)$, $j=1,...,r$.
\end{proof}

\section{Bi-Lipschitz Invariance of the multiplicity}\label{section:mainresults}
Let $f\colon \C^n \to \C$ be a homogeneous polynomial with degree $\deg{f}=d$. Notice that, $\psi\colon \C^n\setminus f^{-1}(0)\to \C\setminus \{0\}$ defined by $\psi(x)=f(x)$ is a locally trivial fibration.  Moreover, we can choose, as geometric monodromy, the homeomorphism $h_f\colon F_f\to F_f$ given by $h_f(x)=e^{\frac{2\pi i}{d}}\cdot x$, where $F_f:=f^{-1}(1)$ is the (global) Milnor fiber of $f$ (see \cite{Milnor:1968}, \S 9). Recall that if $f$ has an isolated singularity at origin $0\in\C^n$, then the Euler characteristic of $F_f$ is given by 
\begin{equation}\label{euler-characteristic}
\chi(F_f)=1+(-1)^{n-1}(d-1)^n
\end{equation}

The next result shows the metric questions \~A1 and \~A2 are equivalent. In other words, to solve the Question \~A1, it is enough work on irreducible homogeneous polynomials.

\begin{theorem}\label{main-theorem}
Question \~A1 has positive answer if, and only if, Question \~A2 has  positive answer.
\end{theorem}

\begin{proof} Obviously, we just need to prove that positive answer to Question \~A2 implies positive answer to Question \~A1.   Let $f,g\colon (\C^n,0)\to (\C,0)$ be two reduced analytic function-germs. Let us suppose that $X=V(f)$, $Y=V(g)$ and $\varphi\colon (\C^n,X,0)\to (\C^n,Y,0)$ is a bi-Lipschitz homeomorphism. Let us denote  by $X_1,\dots,X_r$ and $Y_1,\dots,Y_s$ the irreducible components of the cones $C(X)$ and $C(Y)$ respectively. It comes from Proposition \ref{multiplicities} that $r=s$ and the bi-Lipschitz homeomorphism $d\varphi \colon (\C^n,C(X),0)\rightarrow(\C^n,C(Y),0)$, up to re-ordering of indices, sends $X_i$ onto $Y_i$ and $\kappa_X(X_i)=\kappa_Y(Y_i)$ $\forall$ $i$.

We know that  $X_i$ and $Y_i$ are zero sets of irreducible homogeneous polynomials $f_i$ and $g_i$ respectively. Since, Question \~A2 has positive answer, we get $m(X_i)=m(Y_i)$ $\forall \ i$. Finally, using Eq. \ref{kurdyka-raby}, we get $m(X)=m(Y)$.
\end{proof}

\subsection{Applications of Theorem} \label{subsection:applications}

\bigskip

We denote by $\mathcal{C}$ the set of all complex analytic germs $X\subset \C^n$, at origin $0\in\C^n$,  such that all components of $C(X)$ have isolated singularities.

\begin{theorem}\label{application1}
Let $f,g\colon (\C^n,0)\to (\C,0)$ be two reduced analytic function-germs. Suppose that $V(f)\in \mathcal{C}$.  
If there is a bi-Lipschitz homeomorphism $\varphi\colon (\C^n,V(f),0)\to (\C^n,V(g),0)$, then $V(g)\in \mathcal{C}$ and $m(V(f))=m(V(g))$.
\end{theorem}

\begin{proof} By Theorem \ref{main-theorem}, we can suppose that $f$ and $g$ are irreducible homogeneous polynomials, with degrees $d$ and $e$ respectively, and $f$ has an isolated singularity at origin $0\in\C^n$. As a consequence of Theorem 4.2 in \cite{Sampaio:2014} (see also \cite{Birbrair:2014}, Theorem 3.1), bi-Lipschitz homeomorphisms between two analytic germs send singular subsets onto singular subsets. This is why we claim $g$ has an isolated singularity at origin as well. Now, let us to show $m(V(f),0)=m(V(g),0)$, i.e. $d=e$. It comes from Eq. \ref{euler-characteristic} that 
$$\chi(F_f)=1+(-1)^{n-1}(d-1)^n \quad \mbox{and} \quad \chi(F_g)=1+(-1)^{n-1}(e-1)^n.$$ 
Since $\chi(F_f)=\chi(F_g)$, it follows that $d=e$.
\end{proof}
The next result shows that, in the case of surface singularities in $\C^3$, we do not need any restriction to prove the multiplicity is a bi-Lipschitz invariant.

\begin{theorem}\label{application2} Let $f,g\colon (\C^3,0)\to (\C,0)$ be two reduced analytic function-germs.   
If there is a bi-Lipschitz homeomorphism $\varphi\colon (\C^3,V(f),0)\to (\C^3,V(g),0)$, then $m(V(f))=m(V(g))$.
\end{theorem}

\begin{proof} By Theorem \ref{main-theorem}, we can suppose that $f$ and $g$ are irreducible homogeneous polynomials with degree $d$ and $e$ respectively. By Theorem \ref{application1}, we can suppose that $d,e>1$ and the singular sets $\Sing{V(f)}$ and $\Sing{V(g)}$ are one-dimensional sets. Let us denote by $C_1\dots,C_r$ and $D_1,\dots,D_r$ the irreducible components of $\Sing{V(f)}$ and $\Sing{V(g)}$ respectively. Then, we denote by $b_i(f)$ (respectively $b_i(g)$) the $i$-th Betti number of the Milnor fiber of $f$ (respectively $g$) at the origin, $\mu_j'(f)$ (respectively $\mu_j'(g)$) is the Milnor number of a generic hyperplane slice of $f$ (respectively $g$) at $x_j\in C_j\setminus \{0\}$ (respectively $y_j\in D_j\setminus\{0\}$) sufficiently close to the origin. According to  Theorem 5.11 in \cite{Randell:1979} (see also \cite{Massey:1995}, p. 39, Theorem 3.3 and p. 49, Corollary 4.7), we have the following equations

\begin{equation}\label{le-iomdin}
(d-1)^3=\chi (F_f)-1+d\sum\limits_{i=1}^r\mu_i'(f) \quad \mbox{and} \quad (e-1)^3=\chi (F_g)-1+e\sum\limits_{i=1}^r\mu_i'(g)
\end{equation} 

\bigskip

\noindent{\tt Claim 1.}{\it If $\chi (F_f)=0$, then $d=e$}.

\bigskip
It is valuable to note that we can not skip this step of the proof, because there are examples where $\chi (F_f)=0$, for instance, $f(z_1,z_2,z_3)=z_1z_2z_3$.

\bigskip
\noindent{\it Proof of the Claim 1.} If $\chi (F_f)=0$, then $\chi (F_g)=0$ as well. From Eq. \ref{le-iomdin}, we obtain the following versions of L\^e-Iomdin's formula (see \cite{Iomdin:1974} and \cite{Le:1980}):
$$d^2-3d+3-\sum\limits_{i=1}^r\mu_i'(f)=0 \quad \mbox{and} \quad  e^2-3e+3-\sum\limits_{i=1}^r\mu_i'(g)=0.
$$
On the other hand, according to  \cite{Le:1973b}, Proposition and Th\'eor\`eme 2.3, $q=\sum\limits_{i=1}^r\mu_i'(f)=\sum\limits_{i=1}^r\mu_i'(g)$. Hence, $d$ and $e$ are solutions of the equation 
$$t^2-3t+3-q=0.$$ 
Since this equation has only one solution greater than $1$, it follows that $d=e$.

\hfill{\it End of the proof of the Claim 1.}

\bigskip

From Claim 1, we can suppose that $\chi(F_f)\neq 0$. Thus, $\chi(F_g)\neq 0$ as well.

\bigskip

\noindent{\tt Claim 2.}{\it If $0<k<d$ (respectively $0<k<e$), then $\Lambda(h_f^k)=0$ (respectively $\Lambda(h_g^k)=0$), where $\Lambda(h_f^k)$ (respectively $\Lambda(h_g^k)$) denotes the Lefschetz number of $h_f^k$ (respectively $h_g^k$).}

\bigskip

\noindent{\it Proof of the Claim 2.} We start this proof using the  Topological  Cylindric Structure at Infinity of Algebraic Sets (see \cite{Dimca:1992}, p. 26, Theorem 6.9) to justify that $F=f^{-1}(1)$ has the same homotopy type of $F_R=F\cap \{x\in \C^n;\, \|x\|\leq R\}$, for $R$ large enough. We see the geometric monodromy  $h_f\colon F\to F$ given by $h_f(x)=e^{\frac{2\pi i}{d}}x$, restricted to $F_R$, induces a map $h=h_f|_{F_R}\colon F_R\to F_R$. It is clear $h^k$ does not have fixed point, for $0<k<d$, hence  $\Lambda(h^k)=0$. Since $h_f$ is homotopy equivalent to $h$,  $\Lambda(h_f^k)=0$, for any $0<k<d$. 

\hfill{\it End of the proof of the Claim 2.}

\bigskip

Now, we are ready to finish the proof of the theorem. We know $\chi(F_f)=\chi(F_g)\neq 0$ and, from homotopy invariance of the monodromy (see Theorem 3.3 and remark 3.4 in \cite{Le:1973} or Theorem 1.15 in \cite{Massey:2007}), $\Lambda(h_f^k)=\Lambda(h_g^k)$, for all $k\in \N$. Since, $f$ and $g$ are homogeneous polynomials with degrees $d$ and $e$ respectively, $h_f^d=id \colon  F_f\rightarrow F_f$ and $h_g^e=id \colon  F_g\rightarrow F_g$, hence $\Lambda(h_f^d)=\chi(F_f)\neq 0$ and $\Lambda(h_g^e)=\chi(F_g)\neq 0$. Thus, it follows from Claim 2 that $d=e$.
\end{proof}

\end{document}